\documentclass{article}
\usepackage{graphicx}
\usepackage{amssymb}
\usepackage{amsmath}
\usepackage{amsthm,amsfonts,bbm}
\usepackage{amscd}
\usepackage[all,2cell]{xy}

\UseAllTwocells \SilentMatrices

\newtheorem{thm}{Theorem}[section]
\newtheorem{coro}[thm]{Corollary}
\newtheorem{lem}[thm]{Lemma}
\theoremstyle{definition}
\newtheorem{defi}[thm]{Definition}
\theoremstyle{remark}

\newtheorem{exm}[thm]{\bf Example}
\numberwithin{equation}{section}
\numberwithin{figure}{section}

\def\A{\mathcal{A}}
\def\b{\bar}
\def\B{\mathcal{B}}

\def \Dg{\mathfrak{D}}
\def\diag{{\rm diag}}
\def\I{\mathcal{I}}
\def\la{\lambda}
\def\PS{\mathbb{S}}
\def \PV{\mathbb{V}}
\def \S{\mathbf{S}}
\def \CS{\mathcal{S}}
\def \Spec{\mbox{\rm Spec}}
\def\Z{\mathbb{Z}}

\DeclareMathOperator{\sgn}{sgn}

\begin{document}
\title{\bf The spectral property of hypergraph coverings\thanks{This work was supported by National Natural Science Foundation of China (Grant Nos. 12331012, 11871073,  12171002)}}

\author{Yi-Min Song, Yi-Zheng Fan\thanks{The corresponding author.
Email address: songym@stu.ahu.edu.cn(Y.-M. Song), fanyz@ahu.edu.cn(Y.-Z. Fan), wangy@ahu.edu.cn(Y. Wang), tianmy@stu.ahu.edu.cn(M.-Y. Tian), wanjc@stu.ahu.edu.cn(J.-C. Wan)}, Yi Wang, Meng-Yu Tian, Jiang-Chao Wan\\
{\small \it Center for Pure Mathematics, School of Mathematical Sciences}\\
{\small \it Anhui University, Hefei 230601, P. R. China}}

\date{}

\maketitle

{\small
\noindent
\textbf{Abstract}:
Let $H$ be a connected $m$-uniform hypergraph, and let $\A(H)$ be the adjacency tensor of $H$ whose spectrum is simply called the spectrum of $H$.
Let $s(H)$ denote the number of eigenvectors of $\A(H)$ associated with the spectral radius, and $c(H)$ denote the number of eigenvalues of $\A(H)$ with modulus equal to the spectral radius, which are respectively called the stabilizing index and cyclic index of $H$.
Let $\bar{H}$ be a $k$-fold covering of $H$ which can be obtained from some permutation assignment in the symmetric group $\S_k$ on $H$.
In this paper, we first characterize the connectedness of $\bar{H}$ by its incidence graph and the permutation assignment,
and then investigate the relationship between the spectral property of $H$ and that of $\bar{H}$.
By applying module theory and group representation, if $\bar{H}$ is connected, we prove that $s(H) \mid s(\bar{H})$ and $c(H) \mid c(\bar{H})$.
In particular, when $\bar{H}$ is a $2$-fold covering of $H$,
 if $m$ is even, we show that regardless of multiplicities, the spectrum of $\bar{H}$ contains the spectrum  of $H$ and the spectrum of a signed hypergraph with $H$ as underlying hypergraph;
 if $m$ is odd, we give an explicit formula for $s(\bar{H})$.
We also find some differences on the spectral property between hypergraph coverings and graph coverings by examples.



}

\section{Introduction}
A  \emph{hypergraph} $H=(V,E)$ consists of a vertex set $V=\{v_1,v_2,{\cdots},v_n\}$ denoted by $V(H)$ and an edge set $E=\{e_1,e_2,{\cdots},e_k\}$ denoted by $E(H)$,
 where $e_i \subseteq V$ for $i \in [k]:=\{1,2,\ldots,k\}$.
 If $|e_i|=m$ for each $i \in [k]$ and $m \geq2$, then $H$ is called an \emph{$m$-uniform} hypergraph.
 The hypergraph $H$ is called \emph{simple} if there exists no $i \ne j$ such that $e_i \subseteq e_j$.
In particular, a simple graph is a simple $2$-uniform hypergraph.
For a vertex $v \in V(H)$, denote by $N_H(v)$ or simply $N(v)$ the \emph{vertex neighborhood} of $v$, i.e., the set of vertices of $H$ adjacent to $v$;
and denote by $E_H(v)$ or $E(v)$ the \emph{edge neighborhood} of $v$, i.e. the set of edges containing $v$.
Throughout of this paper, all hypergraphs are considered simple.

A \emph{homomorphism} from a hypergraph $\bar{H}$ to $H$ is a map $\varpi: V(\bar{H})\to V(H)$ such that $\varpi(e) \in E(H)$ for each $e \in E(\bar{H})$; namely, $\varpi$ maps edges to edges.
So $\varpi$ induces a map denoted by $\tilde{\varpi}$ from $E(\b{H})$ to $E(H)$, and particularly $\tilde{\varpi}$ maps $E_{\b{H}}(\bar{v})$ to $E_H(\varpi(\bar{v}))$ for each vertex $\bar{v} \in V(\bar{H})$.

\begin{defi}\label{cov-def}
A homomorphism $\varpi$ from $\bar{H}$ to $H$ is called a \emph{covering projection} if $\varpi$ is a surjection, and the induced map  $\tilde{\varpi}|_{E_{\b{H}}(\bar{v})}: E_{\b{H}}(\bar{v}) \to E_H(v)$ is a bijection for each vertex $v \in V(H)$ and each $\bar{v} \in \varpi^{-1}(v)$.
\end{defi}

Throughout of this paper, \emph{we always assume that the covering projection $\varpi$ in Definition \ref{cov-def} satisfies the following condition: for any edge $e \in E(\b{H})$, $\varpi|_e: e \to \varpi(e)$ is a bijection so that $e$ and $\varpi(e)$ have the same size}.
Under this assumption, if $\bar{H}$ is $m$-uniform, so is $H$.
Suppose that both $\b{H}$ and $H$ are simple graphs in Definition \ref{cov-def}.
Then $\tilde{\varpi}|_{E_{\b{H}}(\bar{v})}: E_{\b{H}}(\bar{v}) \to E_H(v)$ can be replaced by $\varpi|_{N_{\b{H}}(\bar{v})}: N_{\b{H}}(\bar{v}) \to N_H(v)$ as each edge of $E_{\b{H}}(\bar{v})$ (respectively, $E_H(v)$) contains exactly two vertices: $\bar{v}$ and one neighbor of $\bar{v}$ (respectively, $v$ and one neighbor of $v$).

The covering projection $\varpi$ is a surjective homomorphism from $\bar{H}$ to $H$ which preserves the local vertex-edge incidences.
%; and if $\bar{G}$ is $m$-uniform, so is $G$.
If $H$ is connected, then there exists a positive integer $k$ such that each vertex $v$ of $H$ has
$k$ vertices  in its preimage $\varpi^{-1}(v)$, and each edge $e$ of $H$ has $k$ edges in $\varpi^{-1}(e)$.
In this case, $\b{H}$ is called a \emph{$k$-fold covering} (or $k$-sheeted covering) of $H$.
We define an equivalence relation $\sim$ on $V(\b{G})$ induced by $\varpi$ such that $\b{u} \sim \b{v}$ if $\varpi(\b{u})=\varpi(\b{v})$,
then we have the quotient set $V(\b{G})/\varpi:=\{[\b{u}]: \b{u} \in V(\b{G}\}$, where $[\b{u}]$ is an equivalence class of $\b{u}$ under the above relation.
The \emph{quotient hypergraph} of $\b{G}$ by $\varpi$, denoted by $\b{G}/\varpi$, is the hypergraph for which the vertex set is $V(\b{G})/\varpi$ such that $\{[\b{u}_1],\ldots,[\b{u}_t]\}$ forms an edge if there exist $\b{v}_1 \in [\b{u}_1],\ldots,\b{v}_t\in [\b{u_t}]$ such that $\{\b{v}_1,\ldots,\b{v}_t\} \in E(\b{G})$.
By the definition of covering projection, $\b{G}/\varpi$ is isomorphic to $G$.

Gross and Tucker \cite{GT} showed that all coverings of simple graphs can be characterized by the derived graphs of permutation voltage graphs.
Stark and Terras \cite{StaTer} showed the (Ihara) zeta function of a finite graph divides the zeta function of any covering over the graph.
Li and Hou \cite{LH} applied Gross and Tucker's method to generate all hypergraph coverings, and proved that the zeta function of a finite hypergraph divides the zeta function of any covering over the hypergraph.
In \cite{StaTer}, \cite{LH} and related references, the adjacency matrix of a graph was used for discussing zeta functions, and the spectra of a graph and its coverings were investigated for zeta function or Ramanujan graphs.
In particular, Mizuno and Sato \cite{MS} presented a formula for the characteristic polynomial of the derived covering of a simple graph with voltages in any finite group.

In this paper we will investigate the relationship between the spectral property of a uniform hypergraph and that of its coverings.
However, the spectrum of a uniform hypergraph here is not referring to the adjacency matrix \cite{FWL}, Laplacian operator \cite{Chung}, or Laplacian matrix \cite{Rod}.
We will use the  tensor (also called hypermatrix) for the representation of a uniform hypergraph.
Formally, a \emph{tensor} $\A=(a_{i_{1} i_2 \ldots i_{m}})$ of order $m$ and dimension $n$ over $\mathbb{C}$ refers to a
 multiarray of entries $a_{i_{1}i_2\ldots i_{m}}\in \mathbb{C}$ for all $i_{j}\in [n]$ and $j\in [m]$, which can be viewed to be the coordinates of the classical tensor (as a multilinear function) under a certain basis.
In 2005 Lim \cite{Lim} and Qi \cite{Qi} introduced the eigenvalues of tensors independently.
In 2012 Cooper and Dutle \cite{CD} introduced the adjacency tensor of a uniform hypergraph, and applied the eigenvalues of the tensor to characterize the structural property of the hypergraph.

\begin{defi}[\cite{CD}]\label{def-adj}
Let $G$ be an $m$-uniform hypergraph on $n$ vertices $v_1,v_2,\ldots,v_n$.
The \emph{adjacency tensor} of $G$ is defined as $\mathcal{A}(G)=(a_{i_{1}i_{2}\ldots i_{m}})$, an $m$-th order $n$-dimensional tensor, where
$$a_{i_1 i_2 \ldots i_m}=\left\{
\begin{array}{cl}
\frac{1}{(m-1)!}, & \mbox{~if~} \{v_{i_1},\ldots,v_{i_m}\} \in E(H);\\
0, & \mbox{~else}.
\end{array}\right.
$$
\end{defi}

Let $\A$ be a weakly irreducible nonnegative tensor of order $m$.
By the Perron-Frobenius theorem of  nonnegative tensors \cite{CPZ1,FGH,YY1,YY2,YY3},
the spectral radius $\rho(\A)$ is an eigenvalue of $\A$ associated with a unique positive eigenvector up to a scalar, called the \emph{Perron vector} of $\A$.
If $m \ge 3$, including the Perron vector, $\A$ can have more than one eigenvector associated with $\rho(\A)$, which is different from the case of nonnegative irreducible matrices (of order $m=2$).
Fan et al. \cite{FBH} introduced the stabilizing index of a general tensor,
and showed that the number of eigenvector of $\A$ associated with $\rho(\A)$ is exactly the stabilizing index of $\A$.
Recently Fan et al. \cite{FHBproc} proved that there are finitely many eigenvectors of $\A$ associated with the spectral radius up to a scalar.
If the tensor $\A$ has $k$ eigenvalues with modulus equal to $\rho(\A)$, then those $k$ eigenvalues are equally distributed on a circle centered at the origin.
The number $k$ is called the cyclic index of $\A$ \cite{CPZ2}.
For the matrix case, the cyclic index is also called the index of imprimitivity or the index of cyclicity.
Fan et al. \cite{FHB} used the generalized traces of a tensor to give an explicit formula for the cyclic index.

The stabilizing index and the cyclic index of a connected hypergraph $G$, denoted by $s(G)$ and $c(G)$ respectively, are referring to its adjacency tensor.
In this paper, for a connected $m$-uniform hypergraph $G$ and its connected covering $\bar{G}$,
we show that $s(G) \mid s(\bar{G})$ and $c(G) \mid c(\bar{G})$.
 In the situation that $\bar{G}$ is a $2$-fold covering of $G$,
  if $m$ is even, we prove that regardless of multiplicities, the spectrum of $\A(\bar{G})$ contains the spectrum  of $\A(G)$ and the spectrum of a signed hypergraph with $G$ as underlying hypergraph;
  if $m$ is odd, we give an explicit formula for $s(\bar{G})$.
We also find some differences on the spectral property between hypergraph coverings and graph coverings.

\section{Preliminaries}

\subsection{Tensors and hypergraphs}
We first introduce some notions of tensors and hypergraphs.
Let $\A=(a_{i_{1} i_2 \ldots i_{m}})$ be a real tensor of order $m$ and dimension $n$.
The tensor $\A$ is \emph{nonnegative} if all of its entries are nonnegative, and is \emph{symmetric} if all entries $a_{i_1i_2\cdots i_m}$ are invariant under any permutation of its indices.
The digraph $D(\A)$ associated with $\A$ is a digraph on vertices $1,2,\ldots,n$ which has arcs $(i_1, i_2), \ldots, (i_1,i_m)$ for each nonzero entries $a_{i_{1} i_2 \ldots i_{m}}$ of $\A$.
The tensor $\A$ is called \emph{weakly irreducible} if $D(\A)$ is strongly connected \cite{FGH}.
Obviously, the adjacency tensor $\A(G)$  is nonnegative and  symmetric, and it is weakly irreducible if and only if $G$ is connected \cite{PZ,YY3}.

 Given a vector $x\in \mathbb{C}^{n}$, $\A x^{m-1} \in \mathbb{C}^n$, which is defined as follows:
   \[
      (\A x^{m-1})_i =\sum_{i_{2},\ldots,i_{m}\in [n]}a_{ii_{2}\ldots i_{m}}x_{i_{2}}\cdots x_{i_m}, i \in [n].
  \]
 Let $\mathcal{I}=(i_{i_1i_2\ldots i_m})$ be the {\it identity tensor} of order $m$ and dimension $n$, that is, $i_{i_{1}i_2 \ldots i_{m}}=1$ if
   $i_{1}=i_2=\cdots=i_{m} \in [n]$ and $i_{i_{1}i_2 \ldots i_{m}}=0$ otherwise.

\begin{defi}[\cite{Lim,Qi}]\label{eigen} Let $\A$ be an $m$-th order $n$-dimensional tensor.
For some $\lambda \in \mathbb{C}$, if the polynomial system $(\lambda \mathcal{I}-\A)x^{m-1}=0$, or equivalently $\A x^{m-1}=\lambda x^{[m-1]}$, has a solution $x\in \mathbb{C}^{n}\backslash \{0\}$,
then $\lambda $ is called an \emph{eigenvalue} of $\A$ and $x$ is an \emph{eigenvector} of $\A$ associated with $\lambda$,
where $x^{[m-1]}:=(x_1^{m-1}, x_2^{m-1},\ldots,x_n^{m-1})$.
\end{defi}

The \emph{determinant} of $\A$, denoted by $\det \A$, is defined as the resultant of the polynomials $\A x^{m-1}$ \cite{Ha},
and the \emph{characteristic polynomial} $\varphi_\A(\la)$ of $\A$ is defined as $\det(\la \I-\A)$ \cite{Qi,CPZ2}.
It is known that $\la$ is an eigenvalue of $\A$ if and only if it is a root of $\varphi_\A(\la)$.
The \emph{spectrum} of $\A$, denoted by $\Spec(\A)$, is the multi-set of the roots of $\varphi_\A(\la)$.
The \emph{spectral radius} $\rho(\A)$ of $\A$ is the largest  modulus of the eigenvalues of $\A$.
The \emph{spectrum}, \emph{spectral radius}, \emph{eigenvalues} and \emph{eigenvectors} of $G$ are referring to its adjacency tensor $\A(G)$, and the spectral radius of $G$ is denoted by $\rho(G)$.

Let $\mathbb{P}^{n-1}$ be the complex projective space of dimension $n-1$, and let
 $\la$ be an eigenvalue of a tensor $\A$ with dimension $n$.
 The projective variety
$$\PV_\la=\PV_\la(\A):=\{x \in \mathbb{P}^{n-1}: \A x^{m-1}=\la x^{[m-1]}\}$$
is called the \emph{projective eigenvariety} of $\A$ associated with $\la$ \cite{FBH}.
In this paper the \emph{number of eigenvectors} of $\A$ is considered in $\PV_\la(\A)$, i.e.
only one representative vector of the projective equivalence class is counted.

For a matrix $B \in \Z_m^{k \times n}$,
there exist invertible matrices $P \in \Z_m^{k \times k}$ and $Q \in \Z_m^{n \times n}$ such that
\begin{equation} \label{smith}
PBQ=\left(
\begin{array}{ccccccc}
d_1 & 0 & 0 &  & \cdots & & 0\\
0 & d_2 & 0 &  & \cdots & &0\\
0 & 0 & \ddots &  &  & & 0\\
\vdots &  &  & d_r &  & & \vdots\\
 & & & & 0 & & \\
  & & & &  & \ddots & \\
0 &  &  & \cdots &  & &0
\end{array}
\right),
\end{equation}
where $ 0 \le r \le \min\{k,n\}$, $d_i \mid d_{i+1}$ for $i \in [r-1]$, and $d_i \mid m$ for all $i \in [r]$.
The matrix in (\ref{smith}) is called the {\it Smith normal form} of $B$ over $\Z_m$,
where $d_1, \ldots, d_r$ are the \emph{invariant divisors} of $B$ over $\Z_m$.

Let $\A=(a_{i_1i_2\cdots i_m})$ be a symmetric tensor of order $m$ and dimension $n$.
Set $$E(\A):=\{(i_1, i_2, \cdots, i_m)\in [n]^m: a_{i_1i_2\cdots i_m}\neq 0, 1\le i_1\le \cdots \le i_m \le n\}.$$
The \emph{incidence matrix} \cite{FBH} of $\A$  is defined to be a matrix $Z(\A)=(z_{e,j})$ such that
\[z_{e,j}:=|\{k: i_k=j, e=(i_1, i_2, \cdots, i_m) \in E(\A), k \in [m]\}|, e\in E(\A), j \in [n].\]

Let $G$ be a hypergraph.
The \emph{dual} of $G$, denoted by $G^d$, is the hypergraph for which the vertex set is exactly the edge set $E(G)$ of $G$ and edge set is $\{ E_G(v): v \in V(G)\}$.
A \emph{walk} $W$ of length $t$ in $G$ is a sequence of alternate vertices and edges: $v_{0}e_{1}v_{1}e_{2}\ldots e_{t}v_{t}$,
    where $v_{i} \ne v_{i+1}$ and $\{v_{i},v_{i+1}\}\subseteq e_{i}$ for $i=0,1,\ldots,t-1$; and $W$ is called \emph{closed} if $v_0=v_t$.
If $G$ is a simple graph, we simply write $W$ as $v_{0}v_{1}\ldots v_{t}$ as each edge contains exactly two vertices.
The hypergraph $G$ is said to be \emph{connected} if every two vertices are connected by a walk.
The \emph{incidence matrix} of $G$, denoted by $Z(G)=(z_{e,v})$, coincides with that of $\A(G)$, that is, $z_{e,v}=1$ if $v \in e$, and $z_{e,v}=0$ otherwise.

\subsection{Stabilizing index}

The Perron-Frobenius theorem was generalized from nonnegative matrices to nonnegative tensors by Chang et al. \cite{CPZ1},
Yang and Yang \cite{YY1,YY2,YY3}, and Friedland et al. \cite{FGH}.
Here we list parts of the theorem.

\begin{thm}\label{PF1}~~
Let $\A=(a_{i_1i_2\cdots i_m})$ be a nonnegative weakly irreducible tensor of order $m$ and dimension $n$.

\begin{enumerate}
{\em \item[(1)](\cite{FGH})} The spectral radius $\rho(\A)$ is a unique eigenvalue of $\A$ associated with positive eigenvectors, and all these positive eigenvectors differ by a scalar.

{\em \item[(2)](\cite{YY3})}
If $\A$ has $k$ distinct eigenvalues with modulus equal to $\rho(\A)$, then
these eigenvalues are $\rho(\A)e^{\mathbf{i}\frac{2\pi j}{k}}$, $j=0, 1, \cdots, k-1$, and the spectrum of $\A$ keeps invariant under a rotation of angle $\frac{2\pi}{k}$ (but not a smaller positive angle) of the complex plane, where $\mathbf{i}=\sqrt{-1}$.

{\em \item[(3)](\cite{YY3})} If $\B=(b_{i_1i_2\cdots i_m})$ is an $m$-th order $n$-dimensional tensors with $|\B|\le \A$,
namely, $|b_{i_1 i_2 \ldots i_m}| \le a_{i_1 i_2 \ldots i_m}$ for each $i_j \in [n]$ and $j \in [m]$, then $\rho(\B)\le \rho(\A)$.
Moreover, if $\rho(\B)=\rho(\A)$, where $\la=\rho(\A)e^{\mathbf{i}\theta}$ is
an eigenvalue of $\B$ corresponding to an eigenvector $y$, then $y=(y_1, \cdots, y_n)$ contains no zero entries, and $\A=e^{-\mathbf{i}\theta}D^{-(m-1)}\B D$,
where $D=\diag(\frac{y_1}{|y_1|}, \ldots,\frac{y_n}{|y_n|})$.

\end{enumerate}

\end{thm}

In Theorem \ref{PF1}(3), the tensor (or product) $D^{-(m-1)}\B D$ was defined in \cite{Shao}, and has the same spectrum as $\B$ which was proven also in \cite{Shao}.
If $\A$ is further symmetric, Theorem \ref{PF1}(1) can be weakened to  some extent.

\begin{lem}[\cite{Qi13}]\label{sympos}
If $\A$ is a nonnegative symmetric tensor with a positive eigenvector $x$, then $x$ is necessarily associated with $\rho(\A)$.
\end{lem}

In Theorem \ref{PF1}(3), if taking $\B=\A$ and $y$ an eigenvector of $\A$ associated with $\rho(\A)$, then
$$ \A=D_{y}^{-(m-1)}\A D_{y},$$
where $D_{y}:=\diag(\frac{y_1}{|y_1|}, \ldots,\frac{y_n}{|y_n|})$.
In general,
let $\A$ be a tensor of order $m$ and dimension $n$.
Denote
\begin{equation}\label{D0}
\Dg^{(0)}(\A):=\{D: D^{-(m-1)}\A D=\A, d_{11}=1\},
\end{equation}
where $D=\diag(d_{11},\cdots,d_{nn})$ is an $n \times n$ invertible diagonal matrix such that $d_{11}=1$.
It was shown that $\Dg^{(0)}(\A)$ is an abelian group under the usual matrix multiplication, and it is a stabilizer of $\A$ under a certain permutation action; see \cite[Lemmas 2.5-2.6]{FBH}.

\begin{defi}[\cite{FBH}]
For a general tensor $\A$, the cardinality of the abelian group $\Dg^{(0)}(\A)$, denoted by $s(\A)$, is called the \emph{stabilizing index} of $\A$.
\end{defi}

Suppose that $\A$ is nonnegative and weakly irreducible.
 By assigning a quasi-Hadamard product $\circ$ in $\PV_{\rho(\A)}$,
$(\PV_{\rho(\A)},\circ)$ is an abelian group isomorphic to $(\Dg^{(0)}(\A), \cdot)$; see \cite[Lemma 3.1]{FBH}.
So \emph{$s(\A)$ is exactly the number of eigenvectors of $\A$ associated with $\rho(\A)$}.
Assume further that $\A$ is symmetric.
By \cite[Lemma 2.5]{FBH}, $D^m=I$ for each $D \in \Dg^{(0)}(\A)$.
Then $\PV_{\rho(\A)}$ and $\Dg^{(0)}(\A)$ both admit $\Z_m$-modules and are isomorphic to each other, which are also isomorphic to the following $\Z_m$-module:
$$\PS_0(\A):=\{x \in \Z_m^n: Z(\A)x=\mathbf{0} \hbox{~over~} \Z_m, x_1=0\}.$$

Suppose further $\A=\A(G)$ for a connected $m$-uniform hypergraph $G$.
The \emph{stabilizing index} of $G$, denoted by $s(G)$, is referring to the adjacency tensor $\A(G)$.

\begin{thm}[\cite{FBH}, Lemma 3.3, Theorem 3.4, Theorem 3.6] \label{stru}
Let $G$ be a connected $m$-uniform hypergraph on $n$ vertices.
Suppose that the incidence matrix $Z(G)$ has a Smith normal form over $\Z_m$ as in (\ref{smith}).
Then $1 \le r \le n-1$, and as $\Z_m$-modules
\begin{equation}\label{struEQ}
\PV_{\rho(G)}(\A(G)) \cong \Dg^{(0)}(\A(G))\cong \PS_0(\A(G))\cong \oplus_{i, d_i \ne 1} \Z_{d_i} \oplus (n-1-r)\Z_m.
\end{equation}
In particular, $s(G)=|\PS_0(\A(G))|=m^{n-1-r}\Pi_{i=1}^r d_i$.
\end{thm}

\subsection{Cyclic index}
Let $\A$ be a nonnegative weakly irreducible tensor.
The number of distinct eigenvalues of $\A$ with modulus equal to $\rho(\A)$ is called the \emph{cyclic index} of $\A$ by Chang et al. \cite{CPZ2}, denoted by $c(\A)$.
By Theorem \ref{PF1}(2-3), $$\Spec(\A)=e^{\mathbf{i} \frac{2\pi}{c(\A)}}\Spec(\A).$$
So, $c(\A)$ reflects the spectral symmetry of $\A$.
Fan et al. \cite{FHB} defined the spectral symmetry for a general tensor.

\begin{defi}[\cite{FHB}]\label{ell-sym}
Let $\A$ be a general tensor, and let $\ell$ be a positive integer.
The tensor $\A$ is called \emph{spectral $\ell$-symmetric} if
\begin{equation}\label{sym-For}
\Spec(\A)=e^{\mathbf{i} \frac{2\pi}{\ell}}\Spec(\A).
\end{equation}
The maximum $\ell$ such that (\ref{sym-For}) holds is called the \emph{cyclic index} of $\A$,  denoted by $c(\A)$.
\end{defi}

If $\A$ is nonnegative and weakly irreducible, the cyclic index $c(\A)$ in Definition \ref{ell-sym} is consistent with that defined by Chang et al. \cite{CPZ2}
 by Theorem \ref{PF1}(2).
 It is proved that if $\A$ is spectral $\ell$-symmetric, then $\ell \mid c(\A)$ \cite{FHB}.
The spectral symmetry of a connected uniform hypergraph $G$ is referring to $\A(G)$, which can be characterized by the $(m,\ell)$-coloring of $G$.

\begin{defi}[\cite{FHB}]\label{spe-ell-sym-graph}
Let $m \ge 2$ and $\ell \ge 1$ be integers such that $ \ell \mid  m$.
An $m$-uniform hypergraph $G$ is called \emph{$(m,\ell)$-colorable}
if there exists a map $\phi: V(G) \to [m]$ such that if $\{v_{i_1},\ldots, v_{i_m}\} \in E(G)$, then
\begin{equation}\label{gen-col} \phi(v_{i_1})+\cdots+\phi(v_{i_m}) \equiv \frac{m}{\ell} \mod m.\end{equation}
\end{defi}

\begin{thm}[\cite{FHB}]\label{ml-color-G}
Let $G$ be a connected $m$-uniform hypergraph.
Then $G$ is spectral $\ell$-symmetric if and only if $G$ is $(m,\ell)$-colorable.
\end{thm}

Note that Eq. (\ref{gen-col}) is equivalent to
\[
Z(G) \phi =\frac{m}{\ell} \mathbf{1} \hbox{~over~} \Z_m,
\]
where $\phi=(\phi(v_1),\ldots,\phi(v_n))^\top$ is considered as a column vector,
and $\mathbf{1}$ is an all-one vector whose size can be implicated by the context.
Therefore, Theorem \ref{ml-color-G} can be rewritten as follows.

\begin{coro} \label{sym-Zm}
Let $G$ be a connected $m$-uniform hypergraph.
Then $G$ is spectral $\ell$-symmetric if and only if the equation
\begin{equation}\label{ell-Zm}
Z(G) x =\frac{m}{\ell} \mathbf{1} \hbox{\rm~over~} \Z_m
\end{equation} has a solution.
\end{coro}

Hence, the cyclic index $c(G)$ \emph{is exactly the maximum divisor $\ell$ of $m$ such that the equation (\ref{ell-Zm}) has a solution}.
The stabilizing index and cyclic index of hypergraphs were discussed in \cite{FLW,FTL}.

\section{Hypergraph covering and its connectedness}
Let $G$ be a connected graph, and let $\bar{G}$ be a $k$-fold covering of $G$.
In the work of Gross and Tucker \cite{GT}, the covering graph $\bar{G}$ of $G$ is not required to be connected, though in topology the covering space $\bar{G}$ and base space $G$ should be both connected.
In this section, we will investigate the connectedness of hypergraph coverings as
 Theorem \ref{stru} for the stabilizing index and Corollary \ref{sym-Zm} for the cyclic index require the hypergraphs under discussion to be connected.

\subsection{Covering}
Gross and Tucker \cite{GT} used permutation voltage graphs to characterize the coverings of simple graphs.
Let $\S_k$ be the symmetric group on the set $[k]$, and let $D$ be a digraph possibly with multiple arcs.
Let $\phi: E(D) \to \S_k$ which assigns a permutation to each arc of $D$.
The pair $(D, \phi)$ is called a \emph{permutation voltage digraph}.
A \emph{derived digraph} $D^\phi$ associated with $(D, \phi)$ is a digraph with vertex set
$V(D) \times [k]$ such that $((u,i),(v,j))$ is an arc of  $D^\phi$ if and only if $(u,v) \in E(D)$ and $i=\phi((u,v))(j)$.

Let $\overleftrightarrow{G}$ be the symmetric digraph of a simple (undirected) graph $G$, which is obtained from $G$ by replacing each edge $\{u,v\}$ by two arcs with opposite directions, written as $e=(u,v)$ and $e^{-1}:=(v,u)$.
Let  $\phi: E(\overleftrightarrow{G}) \to \S_k$ be a permutation assignment on $\overleftrightarrow{G}$ which holds that $\phi(e)^{-1}=\phi(e^{-1})$ for each arc $e$ of $\overleftrightarrow{G}$.
The pair $(G,\phi)$ is called a \emph{permutation voltage graph}.
The derived digraph $\overleftrightarrow{G}^\phi$, simply written as $G^\phi$, has symmetric arcs by definition, and is considered as a graph.
Gross and Tucker \cite{GT} established a relationship between $k$-fold coverings and derived graphs.

\begin{lem}[\cite{GT}]\label{sim-cov}
Let $G$ be a connected graph and let $\bar{G}$ be a $k$-fold covering of $G$.
Then there exists an  assignment $\phi$ of permutations in $\S_k$ on $G$ such that
$G^\phi$ is isomorphic to $\bar{G}$.
\end{lem}

\begin{lem}\label{comp} Let $G$ be a connected simple graph and let $\bar{G}$ be a $k$-fold covering of $G$.
Then $\bar{G}$ has at most $k$ connected components, each of which is a $\tilde{k}$-fold covering of $G$ for some positive integer $\tilde{k}$, where $1 \le \tilde{k} \le k$.
If $\bar{G}$ has exactly $k$ connected components, then each connected component is isomorphic to $G$.
\end{lem}

\begin{proof}
By Lemma \ref{sim-cov}, there exists a permutation voltage assignment $\phi: E(\overleftrightarrow{G}) \to \S_k$ such that $G^\phi$ is isomorphic to $\bar{G}$.
So it suffices to discuss the graph $G^\phi$.

Let $G^1, \ldots,G^r$ be all connected components of $G^\phi$, where $r \ge 1$.
We assert that each connected component $G^i$ is a covering of $G$ for $i \in [r]$.
Let $\varpi: V(G^\phi) \to V(G)$ be a covering projection of $G^\phi$ on $G$, where $\varpi(v,i)= v$ for each  $(v,i) \in V(G) \times [k]$.
Consider the map $\varpi|_{V(G^i)}: V(G^i) \to V(G)$.
For each vertex $v \in \varpi(V(G^i))$ and each vertex $\bar{v} \in \varpi^{-1}(v) \cap V(G^i)$, surely
$\tilde{\varpi}|_{E_{G^i}(\bar{v})}: E_{G^i}(\bar{v}) \to E_G(v)$ is a bijection by the definition of covering projection.
So, it suffices to prove that $\varpi|_{V(G^i)}: V(G^i) \to V(G)$ is a surjection,
namely, $\varpi(V(G^i))=V(G)$.
Let $u \in \varpi(V(G^i))$ and let $(u,t) \in \varpi^{-1}(u) \cap V(G^i)$.
For any vertex $v \in V(G)$, as $G$ is connected, there exists a walk of $G$: $u_0 u_1 \ldots u_s$, which connects $u$ and $v$, where $u_0=u, u_s=v$.
By the definition of $G^\phi$, there exists a walk in $G^i$ as follows:
$$ (u_0,t)(u_1, \phi(u_1,u_0)(t)) \ldots (u_s, \phi(u_s,u_{s-1})\phi(u_{s-1},u_{s-2})\cdots \phi(u_1,u_0)(t)).$$
Note $\varpi(u_s, \phi(u_s,u_{s-1})\phi(u_{s-1},u_{s-2})\cdots \phi(u_1,u_0)(t))=u_s=v$.
So $v \in  \varpi(V(G^i))$, which implies that
$\varpi|_{V(G_i)}: V(G^i) \to V(G)$ is a surjection.
As $G$ is connected, $G_i$ is a $k_i$-fold covering of $G$ for some positive integer $k_i$.
As $G^\phi$ is a $k$-fold covering of $G$, we have
$\sum_{i=1}^r k_i=k$, which implies that $1 \le k_i \le k$ and $1 \le r \le k$.
So, $G^\phi$ has at most $k$ connected components.
If $r=k$, then each $k_i$ equals $1$, and $G^1,\ldots,G^k$ are all  copies of $G$.
\end{proof}

%For a particular vertex $u \in V(G)$, let $G^1, \ldots,G^r$ be the connected components of $G^\phi$ that contain the vertices of $\varpi^{-1}(u)$, where $G^i$ contains exactly $k_i$ vertices of $\varpi^{-1}(u)$ for $i \in [r]$, and $1 \le r \le k$.
%
%For each $e=(u,v) \in E(\overleftrightarrow{G})$
% and each $(u,t) \in V(G^i) \cap \varpi^{-1}(u)$,
%$G^i$ contains an arc $((u,t), (v, \varphi(v,u)t)$.
%Note that $G^i$ contains exactly $k_i$ vertices of $\varpi^{-1}(u)$, so $G^i$ also contains exactly $k_i$ vertices of $\varpi^{-1}(v)$.
%As both $G$ and $G^i$ are connected, the restriction of $\varpi$ on $G^i$ is $k_i$ to $1$, namely, $G^i$ is a $k_i$-fold covering of $G$.

Li and Hou \cite{LH} generalized Lemma \ref{sim-cov} from graphs to hypergraphs by using two kinds of graph representations of hypergraphs.
Here we only introduce the incidence graph representation of a hypergraph.
The \emph{incidence graph} $B_H$ of a hypergraph $H$ is a bipartite graph with vertex set $V(H) \cup E(H)$ such that $v \in V(H)$ (called the \emph{vertex-vertex} of $B_H$) is adjacent to $e \in E(H)$ (called the \emph{edge-vertex} of $B_H$) if and only if $v \in e$.
Let $\phi: E(\overleftrightarrow{B_H}) \to \S_k$ be a permutation voltage assignment on $B_H$.
From the derived graph $B_H^\phi$, we can construct a hypergraph denoted by $H_B^\phi$ with vertex set $V(H) \times [k]$ such that for each $e \in E(H)$ and each $i \in [k]$, the set of vertices in $B_H^\phi$ adjacent to $(e,i)$ forms a hyperedge, also denoted by $(e,i)$ in $H_B^\phi$.
Li and Hou \cite{LH}  proved that any $k$-fold covering $\bar{H}$ of $H$ is isomorphic to $H_B^\phi$ for some $\phi$.
We give a proof here to emphasize that if $H$ and $\bar{H}$ are both connected, then there is an isomorphism $\psi$ from $B_H^\phi$ to $B_{\bar{H}}$ which sends $V(H)\times [k]$ to $V(\bar{H})$, and hence $H_B^\phi$ is isomorphic to $\bar{H}$.

\begin{lem}\label{hc2}
Let $\bar{H}$ be a connected $k$-fold covering of a connected hypergraph $H$.
Then there exists a permutation assignment $\phi$ in $\S_k$ on $B_H$ such that $B_H^\phi$ is isomorphic to $B_{\bar{H}}$ by a map which sends $V(H)\times [k]$ to $V(\bar{H})$, and hence $H_B^\phi$ is isomorphic to $\bar{H}$.
\end{lem}

\begin{proof}
By definition, there is a $k$ to $1$ surjective covering projection $\varpi: V(\bar{H})\to V(H)$, which induces a $k$ to $1$ surjection $\tilde{\varpi}: E(\bar{H}) \to E(H)$.
So $\varpi$ induces a $k$ to $1$ surjection $\hat{\varpi}$ from $V(B_{\bar{H}})$ to $V(B_H)$, namely,
$$ \hat{\varpi}: V(\bar{H}) \cup E(\bar{H}) \to V(H) \cup E(H), $$
such that $\hat{\varpi}|_{V(\bar{H})}=\varpi$ and $\hat{\varpi}|_{E(\bar{H})}=\tilde{\varpi}$.

We assert that $\hat{\varpi}$ is a covering projection from $B_{\bar{H}}$ to $B_H$, which is necessarily a $k$-fold covering projection by the definition of $\hat{\varpi}$.
Obviously, $\hat{\varpi}$ is a homomorphism from $B_{\bar{H}}$ to $B_H$, as
for any edge $\{\bar{v},\bar{e}\}$ of $B_{\bar{H}}$, surely $\bar{v} \in \b{e}$,
$\varpi(\b{v})\in \varpi(\b{e})$ and hence $\hat{\varpi}(\{\bar{v},\bar{e}\})=\{\varpi(\b{v}), \varpi(\b{e})\} \in E(B_H)$.
To prove $\hat{\varpi}$ preserves the local vertex-edge incidences of graphs, we will use the vertex neighborhoods rather that edge neighborhoods as remarked after Definition \ref{cov-def}.
Note that $N_{B_{\b{H}}}(\bar{v})=E_{\b{H}}(\bar{v})$ for each $\bar{v} \in V(\bar{H})\subseteq V(B_H)$; and $N_{B_{\b{H}}}(\bar{e})=\bar{e}$ for $\bar{e} \in E(\bar{H})\subseteq V(B_H)$, where the first $\bar{e}$ means an edge-vertex of $B_{\b{H}}$ and the second $\b{e}$ means the set of vertices of $\bar{H}$ that are contained in $\b{e}$.
 For each $v \in V(H) \subseteq V(B_H)$ and $\bar{v} \in \varpi^{-1}(v)$,
 the map
$\hat{\varpi}|_{N_{B_{\b{H}}}(\bar{v})}: N_{B_{\b{H}}}(\bar{v}) \to N_{B_{H}}(v)$ is exactly the map $\tilde{\varpi}|_{E_{\b{H}}(\bar{v})}: E_{\b{H}}(\bar{v}) \to E_H(v)$, which is a bijection by definition.
Similarly, for $e \in E(H) \subseteq V(B_H)$ and $\bar{e} \in \tilde{\varpi}^{-1}(e)$,
$\hat{\varpi}|_{N_{B_{\b{H}}}(\bar{e})}: N_{B_{\b{H}}}(\bar{e}) \to N_{B_H}(e)$ is exactly the map $\varpi|_{\b{e}}:  \bar{e} \to e$, which is also a bijection by the assumption on covering projection after Definition \ref{cov-def}.

By Lemma \ref{sim-cov},
there exists a permutation assignment $\phi$ such that $B_H^\phi$ is isomorphic to $B_{\bar{H}}$ via a map $\psi: V(B_H^\phi) \to V(B_{\bar{H}})$.
Note that $B_H^\phi$ is a $k$-fold covering of $B_H$ by a projection $\varphi$ such that
$\varphi|_{\{v\}\times [k]}=v$ and  $\varphi|_{\{e\}\times [k]}=e$ for all $v \in V(H)$ and $e \in E(H)$.
Observe that there exist no two distinct vertex-vertices $(u,i)$ and $(v,j)$ of $B_H^\phi$ such that $\psi(u,i) \in V(\b{H})$ and $\psi(v,j) \in E(\b{H})$.
Otherwise, as $\b{H}$ is connected, $B_{\b{H}}$ and $B_H^\phi$ are both connected,
the distance between $(u,i)$ and $(v,j)$ is even while the distance between $\psi(u,i)$ and $\psi(v,j)$ is odd, which yields a contradiction as $\psi$ is an isomorphism.
Similarly, there exist no two distinct edge-vertices $(e,i)$ and $(f,j)$ of $B_H^\phi$ such that $\psi(e,i) \in V(\b{H})$ and $\psi(f,j) \in E(\b{H})$.
So we can divide the discussion into two cases.

Case 1. $\psi(V(H) \times [k])=V(\bar{H})$ and
$\psi(E(H) \times [k])=E(\bar{H})$. Hence $H_B^\phi$ is isomorphic to $\bar{H}$.

Case 2.   $\psi(V(H) \times [k])=E(\bar{H})$ and
$\psi(E(H) \times [k])=V(\bar{H})$.
Then $B_H^\phi$ is isomorphic to $B_{\bar{H}^d}$ also by $\psi$ which maps the vertex-vertices of $B_H^\phi$ to the vertex-vertices of $B_{\bar{H}^d}$, where $\bar{H}^d$ is the dual of $\bar{H}$.
So the quotient hypergraph $(B_H^\phi)/\varphi$ is isomorphic to $B_{\bar{H}^d}/\varpi$,
 which implies that $B_H$ is isomorphic to $B_{H^d}$ by a map $\tau$ which sends the vertices (edges) of $H$ to the edges (vertices) of $H$.
Let $\tilde{\phi}$ be another permutation assignment on $B_H$ (or $B_{H^d}$) such that $\tilde{\phi}((\tau(v),\tau(e)))=\phi((v,e))$ for each $(v,e) \in V(H) \times E(H)$, where $v \in e$.
It is easily verified that $B_H^\phi$ is isomorphic to $B_{H^d}^{\tilde{\phi}}$ via a map $\bar{\tau}$ such that $\bar{\tau}(v,i)=(\tau(v),i)$ and $\bar{\tau}(e,j)=(\tau(e),j)$ for all vertices $(v,i)$ and $(e,j)$ of $B_H^\phi$.
So, $\psi \circ \bar{\tau}^{-1}$ is an isomorphism from $B_{H^d}^{\tilde{\phi}}$ to $B_{\bar{H}^d}$, and hence an isomorphism from $B_{H}^{\tilde{\phi}}$ to $B_{\bar{H}}$, which sends $V(H)\times [k]$ to $V(\bar{H})$.
The result follows.
\end{proof}

\subsection{Connectedness}
We start the discussion from the connectedness of the graph coverings, and then get the results on hypergraph coverings.
Firstly we introduce some notions for preparation.
A \emph{gain graph} $(G,\mathfrak{G},\phi)$ (also called voltage graph) \cite{Zas0} consists of an underlying graph $G$, a group $\mathfrak{G}$ and a map $\phi: E(\overleftrightarrow{G}) \to \mathfrak{G}$ such that $\phi(e^{-1})=\phi(e)^{-1}$ for each arc $e \in E(\overleftrightarrow{G})$.
So, if $\mathfrak{G}=\S_k$, the gain graph is exactly the permutation voltage graph; and if $\mathfrak{G}=\{z \in \mathbb{C}: |z|=1\}$, the gain graph is called the \emph{complex unit gain graph} \cite{Ref,WGF}.
In particular, a \emph{signed graph} \cite{Zas} is the gain graph by taking $\mathfrak{G} $ to be
the multiplicative subgroup $\{1,-1\}$ of $\mathbb{C}$.
Let $W=v_{0}v_{1}\ldots v_{t}$ be a walk of $(G,\mathfrak{G},\phi)$ (in fact the underlying graph $G$).
The gain value of $W$ is denoted and defined by $\phi(W)=\prod_{i=1}^{t}\phi((v_{i-1},v_i))$.
A cycle $C$ of $G$ is \emph{balanced} if $\phi(C)=1$.
  The gain graph $(G,\mathfrak{G},\phi)$ is \emph{balanced} if each cycle of $G$ is balanced, and is \emph{unbalanced} otherwise.

\begin{thm}\label{connec}
Let $G$ be a connected simple graph and $\phi$ be a permutation  assignment in $\S_k$ on $G$.
Then the following are equivalent.

\begin{enumerate}

\item[\em(1)] $G^\phi$ is connected.

\item[\em(2)] For any $v \in V(G)$ and any $i,j \in [k]$,
there exists a closed walk $W$ of $(G,\phi)$ starting from $v$ such that $i=\phi(W)(j)$.

\item[\em(3)] There exists a vertex $v \in V(G)$ such that for any $i,j \in [k]$, $(G,\phi)$ contains a closed walk $W$ starting from $v$ satisfying $i=\phi(W)(j)$.

\end{enumerate}
\end{thm}

\begin{proof}
(1) $\Rightarrow$ (2).
Suppose $G^\phi$ is connected.
Then for any $v \in V(G)$ and any $i,j \in [k]$,
there exists a walk in $G^\phi$ from $(v,i)$ and $(v,j)$, say
$$W= (v_{0},i_0) (v_{1},i_1)(v_2,i_2)\ldots (v_{t},i_t),$$
 where $(v_0,i_0)=(v,i)$ and $(v_t,i_t)=(v,j)$.
By definition, $(G,\phi)$ contains a closed walk $\tilde{W}=v_{0}v_{1}v_2\ldots v_{t}$, and
$$i=i_0=\phi((v_0,v_1))(i_1)=\phi((v_0,v_1))\phi((v_1,v_2))(i_2)=
\cdots=\phi(W)(i_t)=\phi(W)(j).$$

(2) $\Rightarrow$ (3). It is obvious.

(3) $\Rightarrow$ (1). It suffices to prove for any vertex  $(u,i)$ and a fixed vertex $(v,j)$, there is a walk in $G^\phi$ from $(u,i)$ to $(v,j)$.
As $G$ is connected, there exists a walk from $u$ to $v$, say
$Y=v_{0}v_{1}v_2\ldots v_{t}$, where $v_0 =u$ and $v_t =v$.
Then by definition $(u,i)$ is connecting to $(v,s)$ in $G^\phi$ by a walk $\bar{Y}$, where $s=\phi(Y)^{-1}(i)$.
By the assumption, $(G,\phi)$ contains a closed walk $W$ starting from $v$ such that $s=\phi(W)(j)$.
So $(v,s)$ is connecting to $(v,j)$ in $G^\phi$ by a walk $\bar{W}$.
Hence $(u,i)$ is connecting to $(v,j)$ by joining $\bar{Y}$ and $\bar{W}$.
\end{proof}

We now discuss the other extreme case in Lemma \ref{comp}, that is, each component of $G^\phi$ is a copy of $G$.
Let $(G, \mathfrak{G}, \phi)$ be a gain graph and let $\alpha \in \mathfrak{G}$.
An \emph{$\alpha$-switching} \cite{Zas0} at a vertex $u$ of $(G, \mathfrak{G}, \phi)$ means only                                                         replacing $\phi(e)$ by $\phi^\alpha(e)=\alpha\phi(e)$, and  $\phi(e^{-1})$ by $\phi^\alpha(e^{-1})=\phi(e^{-1})\alpha^{-1}$ for each arc $e=(u,v)$ of $\overleftrightarrow{G}$ starting from $u$.
If $(G, \mathfrak{G},\psi)$ can be obtained from $(G, \mathfrak{G},\phi)$ by a sequence of switchings at some vertices of $G$, then $(G, \mathfrak{G}, \psi)$ is called \emph{switching equivalent} to $(G, \mathfrak{G},\phi)$.

\begin{lem}[\cite{Zas0}]\label{zas}
A gain graph $(G, \mathfrak{G},\phi)$  is balanced if and only if it is switching equivalent to $(G, \mathfrak{G},1)$, where $1$ is the map such that $1(e)=1$ for each arc $e \in E(\overleftrightarrow{G})$.
\end{lem}

In the following we assume that $\mathfrak{G}=\S_k$ and write $(G, \mathfrak{G},\phi)$ simply as $(G,\phi)$.

\begin{lem}\label{swiequ}
Let $G$ be a simple graph.
If $(G, \psi)$ is switching equivalent to $(G, \phi)$, then the derived graph $G^\psi$ is isomorphic to $G^\phi$.
\end{lem}

\begin{proof}
Without loss of generality, assume that $(G, \psi)$ is obtained from $(G, \phi)$ by applying an $\alpha$-switching at a vertex $v$ of $(G, \phi)$.
Define $f: V(G^\phi) \to V(G^\psi)$, which satisfies that
$$f(v,i)=(v, \alpha(i)), f(u,i)=(u,i), i \in [k], u \ne v.$$
It is known that $f$ is a bijection.

If $((v,i),(u,j)) \in E(\overleftrightarrow{G^\phi})$, then $(v,u) \in E(\overleftrightarrow{G})$ and $i=\phi(v,u)(j)$.
So
 $$\alpha(i)=\alpha(\phi(v,u)(j))=\psi(v,u)(j),$$
 and hence $((v,\alpha(i)),(u,j)) \in E(\overleftrightarrow{G^\psi})$, namely, $(f(v,i),f(u,j))\in E(\overleftrightarrow{G^\psi}).$
Conversely, if $((v,\alpha(i)),(u,j)) \in E(\overleftrightarrow{G^\psi})$,
it is easily verified that $((v,i),(u,j)) \in E(\overleftrightarrow{G^\phi})$.
The other cases can be shown similarly.
\end{proof}

Note that $G^1$ is a union of $k$ disjoint copies of $G$, where $1$ is the identity of $\S_k$.

\begin{coro}\label{balTcop}
If $(G, \phi)$ is balanced, or equivalently, $(G, \phi)$ is switching equivalent to $(G, 1)$, then $G^\phi$ is a union of $k$ disjoint copies of $G$.
\end{coro}

Note that for a tree $T$, $(T,\phi)$ is balanced for any $\phi$ by definition.

\begin{coro}
Let $T$ be a tree. Then $T^\phi$ is a union of $k$ disjoint copies of $T$ for any $\phi$.
\end{coro}

We now show the inverse of Corollary \ref{balTcop} is also true.

\begin{thm}\label{kcop}
A permutation voltage graph $(G, \phi)$ is balanced if and only if $G^\phi$ is a union of $k$ disjoint copies of $G$.
 \end{thm}

 \begin{proof}
 It is enough to consider the sufficiency.
 Suppose to the contrary that $(G,\phi)$ is not balanced.
 Then $G$ contains a cycle $C: v_0 v_1 v_2 \ldots  v_t$, where $v_0=v_t$, such that $\phi(C) \ne 1$.
 So there exist two distinct elements $i_0,j_0 \in [k]$ such that $i_0=\phi(C)(j_0)$.
 Let $\bar{G}$ be a connected component of $G^\phi$ which contains the vertex $(v_0,i_0)$.
 By definition, $\bar{G}$ contains a path
 $(v_0,i_0) (v_1, i_1) (v_2,i_2)\ldots (v_t, i_t)$, where $i_{p-1}=\phi(v_{p-1},v_p)(i_p)$ for $p \in [t]$.
 So,
 $$ i_0=\phi((v_0,v_1))\phi((v_1,v_2))\cdots \phi((v_{t-1},v_t))(i_t)=\phi(C)(i_t),$$
 which implies that $i_t=j_0$ and $\bar{G}$ contains both $(v_0,i_0)$ and $(v_0,j_0)$.
Hence, $\bar{G}$ is not a copy of $G$; a contradiction.
 \end{proof}

\begin{thm}\label{2-comp}
Let $G$ be a connected simple graph and let $G^\phi$ be a $2$-fold covering of $G$.
Then the following statements are equivalent.

\begin{enumerate}

\item[\em(1)] $G^\phi$ is connected.

\item[\em(2)] $(G,\phi)$ is unbalanced.

\item[\em(3)] $(G,\phi)$ contains a cycle such that it has an odd number of arcs with $(12)$-permutation assigned by $\phi$.
\end{enumerate}
\end{thm}

\begin{proof}
(1) $\Rightarrow$ (2). It follows by Theorem \ref{kcop}.

(2) $\Rightarrow$ (3). By definition, $(G,\phi)$ contains a cycle $C$ with $\phi(C) \ne 1$. As $\phi(C) \in \S_2$, $\phi(C)=(12)$, which implies that $C$ contains an odd number of arcs with $(12)$-permutation assigned by $\phi$.

(3) $\Rightarrow$ (1). Observe that $(G,\phi)$ contains a cycle $C$ with $\phi(C)=(12)$.
The assertion follows by Theorem \ref{connec}.
\end{proof}

We finally return to the connectedness of hypergraph coverings.
Note a hypergraph $H$ is connected if and only if its incidence graph $B_H$ is connected.
So, by Theorem \ref{connec} and Theorem \ref{2-comp}, we easily get the following results.

\begin{coro}
Let $H$ be a connected $m$-uniform hypergraph, and let $H_B^\phi$ be a $k$-fold covering of $H$, where $\phi$ is a permutation assignment in $\S_k$ on the incidence graph $B_H$.
Then $H_B^\phi$ is connected if and only if there exists a vertex $v$ of $B_H$ such that for any $i,j \in [k]$, $(B_H,\phi)$ contains a closed walk $W$ starting from $v$ satisfying $i=\phi(W)(j)$.
\end{coro}

\begin{coro}\label{2-connect}
Let $H$ be a connected $m$-uniform hypergraph, and let $H_B^\phi$ be a $2$-fold covering of $H$, where $\phi$ is a permutation assignment in $\S_2$ on the incidence graph $B_H$.
Then $H_B^\phi$ is connected if and only if $(B_H,\phi)$ contains a cycle such that it has an odd number of arcs with $(12)$-permutation assigned by $\phi$.

\end{coro}

\section{Stabilizing index and cyclic index of covering}
\subsection{Spectrum}
Let $H$ be a connected $m$-uniform hypergraph on $n$ vertices $v_1,\ldots,v_n$,
and let $B_H$ be its incidence graph with a permutation assignment $\phi$.
From the permutation voltage graph $(B_H, \phi)$, we define
a \emph{signed hypergraph}, denoted by $\Gamma(H,\phi)$, such that for each edge $e \in E(H)$,
$$ \sgn e=\prod_{v \in e} \sgn \phi(v,e).$$
The \emph{adjacency tensor} of $\Gamma(H,\phi)$ is defined as
$\A(\Gamma(H,\phi))=(a^\phi_{i_1 i_2 \ldots i_m})$, where
$$a^\phi_{i_1 i_2 \ldots i_m}=\left\{
\begin{array}{cl}
\frac{\sgn e}{(m-1)!}, & \mbox{~if~} e=\{v_{i_1},\ldots,v_{i_m}\} \in E(H);\\
0, & \mbox{~else}.
\end{array}\right.
$$
The \emph{spectrum}, \emph{eigenvalues} and \emph{eigenvectors} of $\Gamma(H,\phi)$ are referring to $\A(\Gamma(H,\phi))$.

\begin{thm}\label{specont}
Let $H$ be a connected $m$-uniform hypergraph, let $H_B^\phi$ be a $k$-fold covering of $H$, where $\phi$ is a permutation assignment in $\S_k$ on $B_H$.
Then, regardless of multiplicities, the spectrum of $H_B^\phi$ contains the spectrum of $H$, and $\rho(H_B^\phi)=\rho(H)$.

In particular, if $k=2$ and $m$ is even, then, regardless of multiplicities,  the spectrum of $H_B^\phi$ also contains the spectrum of the signed hypergraph $\Gamma(H,\phi)$.
\end{thm}

\begin{proof}
Define \begin{equation}\label{tau}
\tau: \mathbb{C}^{V(H)} \to \mathbb{C}^{V(H)\times [k]}
\end{equation}
such that $\tau(x)|_{\{v\} \times [k]}=x(v)$ for all $x \in \mathbb{C}^{V(H)}$ and all $v \in V(H)$.
Let $x$ be an eigenvector of $\A(H)$ associated with an eigenvalue $\lambda$.
We assert that $\tau(x)$ is an eigenvector of $\A(H_B^\phi)$ also associated with the eigenvalue $\lambda$.
By eigenvector equations, for each $v \in V(H)$,
\[
\lambda x(v)^{m-1}=\sum_{e \in E_H(v)}\prod_{u \in e \backslash \{v\}}x(u).
\]
For each vertex $v \in V(H)$ and each $i \in [k]$,
if $(e,j) \in E_{H_B^\phi}((v,i))$, then $v \in e$ and $j=\phi(e,v)(i)$, and
$E_H(v)$ is one to one mapping onto $E_{H_B^\phi}((v,i))$ by
\begin{equation}\label{eta} \eta_{v,i}: e \mapsto (e, \phi(e,v)(i));\end{equation}
furthermore, if $(u,t) \in (e, \phi(e,v)(i))$, then $u \in e$ and $t=\phi(u,e)\phi(e,v)(i)$, and
$e$ is one to one mapping onto $(e,\phi(e,v)(i))$ by
\begin{equation}\label{zeta}  \zeta_{v,i,e}:u \mapsto (u, \phi(u,e)\phi(e,v)(i)).\end{equation}
So
\begin{align*}
  \lambda (\tau(x)(v,i))^{m-1}& =  \lambda x(v)^{m-1}=\sum_{e \in E_H(v)}\prod_{u \in e \backslash \{v\}}x(u)\\
  &=  \sum_{\eta_{v,i}(e) \in E_{H_B^\phi}((v,i))}\prod_{\zeta_{v,i,e}(u) \in \eta_{v,i}(e) \backslash \{(v,i)\}}\tau(x)(\zeta_{v,i,e}(u)),
\end{align*}
which implies  $\tau(x)$ is an eigenvector of $\A(H_B^\phi)$ associated with the eigenvalue $\lambda$.

In the above discussion, by Theorem \ref{PF1}(1), taking $\lambda=\rho(H)$ and $x$ be a positive eigenvector associated with $\rho(H)$, then $\tau(x)$ is a positive eigenvector of $\A(H_B^\phi)$ associated with the eigenvalue $\rho(H)$.
By Lemma \ref{sympos}, $\rho(H)$ is the spectral radius of $\A(H_B^\phi)$, i.e. $\rho(H_B^\phi)=\rho(H)$.

Now suppose $k=2$ and $m$ is even.
Let $x$ be an eigenvector of $\Gamma:=\Gamma(H,\phi)$ associated with an eigenvalue $\lambda$.
Define $y \in \mathbb{C}^{V(H)\times [2]}$ such that $y(v,1)=x(v)$ and $y(v,2)=-x(v)$ for each $v \in V(H)$.
We will show that $y$ is an eigenvector of $\A(H_B^\phi)$ also associated with the eigenvalue $\lambda$.

By eigenvector equations, for each $v \in V(\Gamma)$,
\[
\lambda x(v)^{m-1}=\sum_{e \in E_\Gamma(v)}\sgn e \prod_{u \in e \backslash \{v\}}x(u).
\]
For each vertex $v \in V(H)$ and each edge $e \in E_\Gamma(v)$, by Eq. (\ref{eta}) and (\ref{zeta}), $\eta_{v,1}(e)=(e, \phi(e,v)(1)) \in E_{H_B^\phi}((v,1))$
and  $\zeta_{v,1,e}(u)=(u, \phi(u,e)\phi(e,v)(1)) \in \eta_{v,1}(e)$.
By definition,
$$ y(\zeta_1(u))=(\sgn \phi(u,e)\phi(e,v)) x(u), $$
and since $m$ is even,
\begin{align*}
\prod_{\zeta_{v,1,e}(u) \in \eta_{v,1}(e) \backslash \{(v,1)\}}y(\zeta_{v,1,e}(u))& =\sgn \phi(e,v)^{m-2} \sgn e \prod_{u \in e\backslash \{v\}} x(u)\\
& =\sgn e \prod_{u \in e\backslash \{v\}} x(u).
\end{align*}
So
\begin{equation}\label{case1}
\begin{split}
  \lambda y(v,1)^{m-1}&=\lambda x(v)^{m-1}=\sum_{e \in E_\Gamma(v)}\sgn e\prod_{u \in e \backslash \{v\}}x(u)\\
  &= \sum_{\eta_{v,1}(e) \in E_{H_B^\phi}((v,1))}\prod_{\zeta_{v,1,e}(u) \in \eta_{v,1}(e) \backslash \{(v,1)\}}y(\zeta_{v,1,e}(u)).
\end{split}
\end{equation}

Similarly, for each vertex  $\zeta_{v,2,e}(u)=(u, \phi(u,e)\phi(e,v)(2)) \in \eta_{v,2}(e)$, where
$\eta_{v,2}(e)=(e, \phi(e,v)(2)) \in E_{H_B^\phi}((v,2))$,
$$ y(\zeta_{v,2,e}(u))=-(\sgn \phi(u,e)\phi(e,v)) x(u),$$
 and $$\prod_{\zeta_{v,2,e}(u) \in \eta_{v,2}(e) \backslash \{(v,2)\}}y(\zeta_{v,2,e}(u))=(-1)^{m-1}\sgn e \prod_{u \in e\backslash \{v\}} x(u)
 =-\sgn e \prod_{u \in e\backslash \{v\}} x(u).$$
So
\begin{equation}\label{case2}
\begin{split}
  \lambda y(v,2)^{m-1}&=-\lambda x(v)^{m-1}=-\sum_{e \in E_\Gamma(v)}\sgn e\prod_{u \in e \backslash \{v\}}x(u)\\
  &= \sum_{\eta_{v,2}(e) \in E_{H_B^\phi}((v,2))}\prod_{\zeta_{v,2,e}(u) \in \eta_{v,2}(e) \backslash \{(v,2)\}}y(\zeta_{v,2,e}(u)).
\end{split}
\end{equation}
By Eq. (\ref{case1}) and (\ref{case2}), we show that  $y$ is an eigenvector of $\A(H_B^\phi)$ also associated with the eigenvalue $\lambda$.
\end{proof}

\begin{coro}
Let $H$ be a connected $m$-uniform hypergraph, let $\bar{H}$ be a connected $k$-fold covering of $H$.
Then, regardless of multiplicities, the spectrum of $\bar{H}$ contains the spectrum of $H$, and $\rho(\bar{H})=\rho(H)$.

In particular, if $k=2$ and $m$ is even, then, regardless of multiplicities,  the spectrum of $\bar{H}$ also contains the spectrum of a signed hypergraph with $H$ as underlying hypergraph.
\end{coro}

\begin{proof}
By Lemma \ref{hc2}, there exists a permutation assignment $\phi$ on $B_H$ such that $\bar{H}$ is isomorphic to $H_B^\phi$.
The result follows by Theorem \ref{specont}.
\end{proof}

By Theorem \ref{specont}, if $H_B^\phi$ is connected, then $s(H) \le s(H_B^\phi)$ by the map $\tau$ defined in Eq. (\ref{tau}).
In fact, $(\PV_{\rho(H)}, \circ)$ can be embedded as  a $\Z_m$-submodule of  $(\PV_{\rho(H_B^\phi)},\circ)$ so that $s(H) \mid s(H_B^\phi)$.
However, it will need more preparations to show the above division relation.
We will use another $\Z_m$-module involved with incidence matrix to investigate the division relation in next subsection.

The $2$-fold covering $\phi$ of a simple graph $G$ is also called a \emph{$2$-lift} of $G$ \cite{MSS}.
The spectrum of a $2$-lift of $G$ is exactly the union of the spectrum of $G$ and the spectrum of $\Gamma(G, \phi)$ \cite{BL}.
However, the above result does not hold for $2$-fold coverings of a hypergraph; see the following example.

\begin{exm}\label{4-uni}
By using the terminology in \cite{KF}, denote by $C_k^{4,2}$ a $4$-uniform hypergraph obtained from a cycle $C_k$ of length $k$ (as a simple graph) by blowing up each vertex into a $2$-set and keeping the adjacency.
Consider $H=C_3^{4,2}$ with vertex set $[6]$ and edge set $\{e_1=\{1,2,3,4\}, e_2=\{3,4,5,6\}, e_3=\{5,6, 1,2\}\}.$

(1) Let $\phi$ be a permutation assignment in $\S_2$ on $B_H$ such that $\phi(e_1,3)=\phi(e_1,4)=(12)$, and $\phi(e,v)=1$ for all other incidences $(e,v)$.
By definition the signed hypergraph $\Gamma(H,\phi)$ is same as $H$ as the sign of each edge of $\Gamma(H,\phi)$ equals $1$.
The $2$-fold covering hypergraph $H_B^\phi$  has the following edges
$$\{11,21,32,42\},\{32,42,52,62\}, \{52,62,12,22\},$$
$$\{12,22,31,41\},\{31,41,51,61\},\{51,61,11,21\},$$
where a vertex $(v,k)$ in the vertex set $V(H)\times [2]$ of $H_B^\phi$ is simply written as $vk$.
Note that $H_B^\phi$ is isomorphic to $C_6^{4,2}$.
By using {\scshape SageMath}\footnote{https://www.sagemath.org} package: TensorCharpolyPackage
written by Aaron Dutle\footnote{https://people.math.sc.edu/dutle/spectraresults.html}, we get the characteristic polynomial of $\A(H)$ as follows:
$$\varphi_{\A(H)}(\la)=\la^{498} (\la^2 - 4)^{32}  (\la^2 - 1)^{208}  (\la^2 + 1)^{48} (\la^2 - \la + 2)^{96}  (\la^2 +\la + 2)^{96}.$$

By Theorem 2.10 in \cite{FKT}, $H_B^\phi$ has an eigenvalue $\sqrt{3}$, which is the largest eigenvalue of a path $P_5$ on $5$ vertices (as an induced subgraph of $C_6$). But $\sqrt{3}$ is not an eigenvalue of $H$ or $\Gamma(H,\phi)$.

(2) Let $\psi$ be another permutation assignment in $\S_2$ on $B_H$ such that $\psi(e_1,3)=\psi(e_2,5)=(12)$, and $\psi(e,v)=1$ for all other incidences $(e,v)$.
Then $\Gamma(H,\psi)$ contains exactly two negative edges, namely $\{1,2,3,4\}$ and $\{3,4,5,6\}$.
The characteristic polynomial of $\A(\Gamma(H,\psi))$ is
$$\varphi_{\A(\Gamma(H,\psi))}(\la)=\la^{1074}(\la^4 - 1)^{32}     (\la^4 - 4)^{64}. $$

The $2$-fold covering hypergraph $H_B^\psi$  has the following edges
$$\{11,21,32,41\},\{31,41,52,61\},\{51,61,11,21\},$$
$$\{12,22,31,42\},\{32,42,51,62\},\{52,62,12,22\}.$$
Let $K$ be the subhypergraph of $H_B^\psi$ with edges
$$\{11,21,32,41\},\{31,41,52,61\},\{51,61,11,21\},\{32,42,51,62\}.$$
Observe that if $\la$ is an eigenvalue of $\A(K)$ associated with an eigenvector $x$, then $\la$ is also an eigenvalue of $A(H_B^\psi)$ associated with an eigenvector $\tilde{x}$ by setting $\tilde{x}|_{V(K)}=x$ and $\tilde{x}(12)=\tilde{x}(22)=0$.
If letting $x(11)=x(21)=:a$, $x(32)=x(41)=x(51)=x(61)=:b$ and $x(31)=x(42)=x(52)=x(62)=:c$, by eigenvector equations, we get that $K$ and hence
$H_B^\psi$ has an eigenvalue $\sqrt{3}$, which is neither an eigenvalue of $H$ nor $\Gamma(H,\psi)$.
\end{exm}

For the $2$-fold covering $H_B^\phi$ of an $m$-uniform hypergraph $H$, if $m$ is odd, the spectrum of $H_B^\phi$ can not contain the spectrum of the signed hypergraph $\Gamma(H,\phi)$; see the following example.

\begin{exm}\label{3-uni}
Let $H$ be a $3$-uniform hypergraph with vertex set $[4]$ and edge set $\{e_1=\{1,2,3\},e_2=\{2,3,4\}\}$.
Let $\phi$ be a permutation assignment in $\S_2$ such that $\phi(e_1,2)=(12)$ and $\phi(e,v)=1$ for all other incidences $(e,v)$.
Then the signed hypergraph $\Gamma(H,\phi)$ contains a negative edge $\{1,2,3\}$ and a positive edge $\{2,3,4\}$,
and $H_B^\phi$ is obtained from $C_4$ by inserting an additional vertex into each edge, called the power hypergraph of $C_4$ \cite{HQS}.
The characteristic polynomial of $\A(\Gamma(H,\phi))$ is
$$ \varphi_{\A(\Gamma(H,\phi))}(\la)=\la^{14}  (\la^6 + 4)^3,$$
and the characteristic polynomials of $\A(H_B^\phi)$ is
$$ \varphi_{\A(H_B^\phi)}(\la)=\la^{493}(\la^3 - 1)^{24} (\la^3 - 4)^{27}  (\la^3 - 2)^{126}.$$
So $\sqrt[3]{2}\mathbf{i}$ is an eigenvalue of $\Gamma(H,\phi)$ but not an eigenvalue of $H_B^\phi$.
\end{exm}

    \begin{coro}
    Let $H$ be a connected $m$-uniform hypergraph, and let $\bar{H}$ be a connected $k$-fold covering of $H$. Then $c(H) \mid c(\bar{H})$.
    \end{coro}

\begin{proof}
By Lemma \ref{hc2}, there exists a permutation assignment $\phi$ on $B_H$ such that $\bar{H}$ is isomorphic to $H_B^\phi$.
So, it suffices to consider $H_B^\phi$.
By Theorem \ref{specont}, regardless of multiplicities, the spectrum of $H_B^\phi$ contains that of $H$, and $\rho(H_B^\phi)=\rho(H)$.
So, by Theorem \ref{PF1}(2) and the definition of cyclic index,
$\rho(H)e^{\mathbf{i}\frac{2\pi}{c(H)}}$ is an eigenvalue of $\A(H)$, and hence
$\rho(H_B^\phi)e^{\mathbf{i}\frac{2\pi}{c(H)}}$ is an eigenvalue of $\A(H_B^\phi)$.
By by Theorem \ref{PF1}(3), there exists a diagonal matrix $D$ such that $$\A(H_B^\phi)=e^{-\mathbf{i}\frac{2\pi}{c(H)}}D^{-(m-1)}\A(H_B^\phi)D.$$
So $\Spec(H_B^\phi)=e^{\mathbf{i}\frac{2\pi}{c(H)}}\Spec(H_B^\phi)$, implying that
$H_B^\phi$ is spectral $c(H)$-symmetric, and hence  $c(H) \mid c(H_B^\phi)$.
\end{proof}

\subsection{$\Z_m$-Module}
We will use $\Z_m$-Module to establish the division relation between $s(H)$ and $s(\bar{H})$, where $\bar{H}$ is a $k$-fold covering of $H$.

\begin{thm}
Let $H$ be a connected $m$-uniform hypergraph, and let $\bar{H}$ be a connected $k$-fold covering of $H$.
Then $\PS_0(H)$ can be embedded as a $\Z_m$-submodule of $\PS_0(\bar{H})$, and hence $s(H) \mid s(\bar{H})$.
\end{thm}

\begin{proof}
By Lemma \ref{hc2}, there exists a permutation assignment $\phi$ on $B_H$ such that $B_H^\phi$ is isomorphic to $B_{\bar{H}}$ via a map $\psi$ with $\psi(V(H)\times [k])=V(\bar{H})$.
So, it suffices to consider $H_B^\phi$ whose incidence graph is exactly $B_H^\phi$.

Consider
$$ \CS_0(H):=\{x \in \Z_m^{V(H)}: Z(H)x=\mathbf{0} \hbox{~over~} \Z_m\},$$
and
$$\CS_0(H_B^\phi):=\{x \in \Z_m^{V(H_B^\phi)}: Z(H_B^\phi)x=\mathbf{0} \hbox{~over~} \Z_m\}.$$
Let
\begin{equation}\label{hpi} \tau: \Z_m^{V(H)} \to \Z_m^{V(H) \times [k]} \end{equation}
such that $\tau(x)|_{\{v\} \times [k]}=x (v)$ for all $x \in \Z_m^{V(H)}$ and $v \in V(H)$.
We assert that for each $x \in \CS_0(H)$, $\tau(x) \in \CS_0(H_B^\phi)$.
If $x \in \CS_0(H)$, then for each edge $e=(v_{i_1}, \ldots,v_{i_m}) \in E(H)$,
$$ x(v_{i_1})+\cdots+x(v_{i_m}) =0 \mod m.$$
Now for each edge $(e,i) \in E(H_B^\phi)$, by definition
$$ (e,i)=\{(v_{i_1}, \phi(v_{i_1},e)(i)), \ldots, (v_{i_m}, \phi(v_{i_m},e)(i))\}.$$
So
\begin{align*}
& \tau(x)(v_{i_1}, \phi(v_{i_1},e)(i))+ \cdots+ \tau(x)(v_{i_m}, \phi(v_{i_m},e)(i))\\
& =
x(v_{i_1})+\cdots+x(v_{i_m}) =0 \mod m,
\end{align*}
which implies that $\tau(x) \in \CS_0(H_B^\phi)$.

Observe that ${\tau}: \CS_0(H) \to \CS_0(H_B^\phi)$ is an injection and also a $\Z_m$-module homomorphism.
So $\CS_0(H)$ is $\Z_m$-isomorphic to ${\tau}(\CS_0(H))$, the latter of which is a
 $\Z_m$-submodule of $\CS_0(H_B^\phi)$.
 Therefore, $\CS_0(H)$ can be embedded as a $\Z_m$-submodule of $\CS_0(H_B^\phi)$.
Note that if $x \in \CS_0(H)$ then $x +t \mathbf{1} \in \CS_0(H)$ for any $t \in \Z_m$ as $Z(H) \mathbf{1}=0 \mod m$. So $\PS_0(H) \cong \CS_0(H)/(\Z_m \mathbf{1})$,  and similarly $\PS_0(H_B^\phi) \cong \CS_0(H_B^\phi)/(\Z_m \mathbf{1})$.
Then $\PS_0(H)$ is embedded as a $\Z_m$-submodule of $\PS_0(H_B^\phi)$.
So we get $s(H) \mid s(H_B^\phi)$ by Theorem \ref{stru}.
\end{proof}

One may wonder what is the exact expression or value of $s(H_B^\phi)$.
We will discuss the problem by the representation theory of group ring.

\subsection{Representation}

Given a group $G$ and a ring $R$, the group ring or group algebra of $G$ over  $R$, denoted and defined by
$RG=\left\{\sum_{g \in G} r_g g \mid r_g \in R\right\},$
which is a free $R$-module with the elements of $G$ as a basis.

Let $M$ be a $R$-module. $M$ is called an $RG$-module, if there exists
a $R$-module homomorphism
$$\varrho: RG \to \text{End}_R(M),$$
where $\text{End}_R(M)$ is the ring of $R$-module endomorphism of $M$.
$(M, \rho)$ is also called an \emph{$R$-module representation} of $G$.
If $M$ is a free $R$-module with a basis $b_1,\ldots,b_k$, then the above $R$-module homomorphism $\varrho$ is equivalent to a group homomorphism
$$ \hat{\varrho}: G \to \text{GL}_k(R),$$
where $ \text{GL}_k(R)$ is the group of all invertible matrices of order $k$ over $R$.
If $gb_j=\sum_{i=1}^k \alpha_{ij}b_i$, then $\hat{\varrho}(g)=(\alpha_{ij})_{k \times k}$.
$\hat{\varrho}$ is called an \emph{$R$-matrix representation} of $G$.

Now consider a special case.
Take $R=\mathbb{Z}_m$ and $M$ be a free $\mathbb{Z}_m$-module with a basis $b_1,\ldots,b_k$.
Let $G$ be a subgroup of $\S_k$.
Consider the \emph{permutation representation} of $G$ over $M$, i.e. each element $g \in G$ maps $b_i$ to $b_{g(i)}$.
Using matrix language,
\begin{equation}\label{perzm} \varrho:G \to \text{GL}_k(\mathbb{Z}_m), g \mapsto P_g=(p^g_{ij}),\end{equation}
where,
\begin{equation}\label{four-eq1}
p_{ij}^{(g)}=\begin{cases}
1, & \mbox{if~} \ i=g(j);\\
0, & \mbox{else}.
\end{cases}
\end{equation}
In this case, we call $\varrho$ the \emph{$\Z_m$-permutation matrix representation} of $G$.

Let $M_1=\mathbb{Z}_m \left(\sum_{i=1}^k b_i\right)$.
It is easily verified that $M_1$ is a subrepresentation of $M$, denoted by $(M_1, \varrho_1)$, where $\rho_1$ is the identity.
Let $$ M_2=\left\{\sum_{i=1}^k \ell_i b_i \mid \ell_i \in \mathbb{Z}_m, \sum_{i=1}^k \ell_i=0\right\}.$$
It is easily verified that $M_2$ is a subrepresentation of $M$, denoted by $(M_2, \varrho_2)$, with a basis $b_1-b_i$, $i=2,\ldots,k$.
Suppose that $\gcd(m,k)=1$.
Then $k$ is invertible in $\mathbb{Z}_m$.
As $$ \sum_{i=1}^k \ell_i b_i=\frac{\sum_{i=1}^k \ell_i}{k} \sum_{i=1}^k b_i + \sum_{i=1}^k \left(\ell_i- \frac{\sum_{i=1}^k \ell_i}{k} \right)b_i \in M_1 + M_2,$$
we have $M=M_1+M_2$.
If $c \in M_1 \cap M_2$, then there exist $\ell, \ell_1, \ldots, \ell_k \in \mathbb{Z}_m$ such that
$$ c=\ell \sum_{i=1}^k b_i =\sum_{i=1}^k \ell_i b_i,$$
where $\sum_{i=1}^k \ell_i=0$.
So we have $k\ell =0$, and thus $\ell=0$ and $c=0$, which implies that
$M=M_1 \oplus M_2$.

\begin{thm}\label{perrep}
Let $M$ be a free $\mathbb{Z}_m$-module of rank $k$, and let
$G$ be a subgroup of $\S_k$, where $\gcd(m,k)=1$.
Then as a $\mathbb{Z}_mG$-module by permutation representation,  $M$ has a decomposition
$$ M =M_1 \oplus M_2,$$
where $M_1$ is a subrepresentation of $M$ with degree $1$, and $M_2$ is a subrepresentation of $M$ with degree $k-1$;
or equivalently, there exists an invertible matrix $T$ over $\mathbb{Z}_m$ such that for all $g \in G$, the representation matrix $P_g$ in (\ref{perzm}) holds that
\begin{equation}\label{rep-dec} T^{-1} P_g T =I_1 \oplus \varrho_2(g),\end{equation}
where $I_1$ is the identity matrix representation of $G$ with order $1$ and $\varrho_2(g)$ is the matrix representation of $G$ with order $k-1$.
\end{thm}

We next give an expression of the incidence matrix of $H_B^\phi$, where
 $\phi: E(\overleftrightarrow{B_H}) \to \S_k$ is a permutation assignment on $B_H$.
Let $\Phi=\langle \phi(e)\mid e\in E(\overleftrightarrow{B_H})\rangle$, the subgroup of $\S_k$ generated by the permutation assigned on $B_H$.
For each $g\in\Phi$, define a matrix $Z_g=(z_{ev}^{(g)})$,
where
\begin{equation}\label{four-eq3}
z_{ev}^{(g)}=\begin{cases}
1, & \mbox{if~} \ (e,v)\in E(\overleftrightarrow{B_H}) \mbox{~and~} \phi(e,v)=g;\\
0, & \mbox{else}.
\end{cases}
\end{equation}
Then the incidence matrix $Z(H)$ holds that
\begin{equation}\label{four-eq4}
Z(H) = \sum\limits_{g\in\Phi}Z_g,
\end{equation}
and the incidence matrix $Z(H_B^\phi)$ satisfies that \begin{equation}\label{four-eq5}
Z(H_B^\phi) = \sum_{g\in\Phi}Z_g\otimes P_g.
\end{equation}

\begin{thm}\label{IncMD}
Let $H$ be a connected $m$-uniform hypergraph on $n$ vertices, let $H_B^\phi$ be a $k$-fold covering of $H$, where $\phi: E(\overleftrightarrow{B_H}) \to \S_k$ is a permutation assignment on $B_H$, and $\gcd(m,k)=1$.
Let $\Phi=\langle \phi(e)\mid e\in E(\overleftrightarrow{B_H})\rangle$, with a $\Z_m$-permutation matrix representation as defined in (\ref{perzm}).
Then there exists an invertible matrix $T$ over $\Z_m$ such that
\begin{equation}\label{incdec} (I \otimes T^{-1}) Z(H_B^\phi) (I \otimes T)=Z(H) \oplus \left(\sum_{g\in\Phi}Z_g\otimes \varrho_2(g)\right),\end{equation}
where $\varrho_2(g)$ is the matrix representation of $\Phi$ with order $k-1$.

Moreover, if  $H_B^\phi$ is connected, and $Z(H,\phi):=\sum_{g\in\Phi}Z_g\otimes\varrho_2(g)$ has invariant divisors $\hat{d}_1,\ldots,\hat{d}_{\hat{r}}$ over $\Z_m$, then
\begin{equation}\label{staform}
 s(H_B^\phi)=s(H) \left(m^{n-\hat{r}} \prod_{i=1}^{\hat{r}}\hat{d}_i\right),
\end{equation}
\end{thm}

\begin{proof}
By Theorem \ref{perrep}, there exists an invertible matrix $T$, such that for any $g \in \Phi$,
\begin{equation}\label{pereq} T^{-1} P_g T =I_1 \oplus \varrho_2(g),\end{equation}
where $\varrho_2(g)$ is the matrix representation of $\Phi$ with order $k-1$.
By Eq. (\ref{four-eq4}), (\ref{four-eq5}) and (\ref{pereq}), we have
\[
\begin{split}
(I \otimes T^{-1}) Z(H_B^\phi) (I \otimes T)&=\sum_{g\in\Phi}Z_g\otimes(T^{-1}P_g T)\\
&=\sum_{g\in\Phi}Z_g\otimes(I_1 \oplus \varrho_2(g))\\
&=Z(H) \oplus \left(\sum_{g\in\Phi}Z_g\otimes\varrho_2(g)\right).
\end{split}
\]

Suppose that $Z(H)$ has invariant divisors $d_1,\ldots,d_r$ over $\Z_m$.
Then $Z(H_B^\phi)$ has invariant divisors $d_1,\ldots,d_r, \hat{d}_1,\ldots,\hat{d}_{\hat{r}}$ over $\Z_m$.
The result follows by Theorem \ref{stru}.
\end{proof}

In Theorem \ref{IncMD}, if we know the representation $\varrho_2$ of $\Phi$, then we will get the matrix  $Z(H,\phi)$ explicitly.
By calculating the Smith normal forms of $Z(H)$ and $Z(H,\phi)$, we will get the exact value of $s(H_B^\phi)$ by Theorem \ref{stru}.
In particular, the $\Z_m$-permutation representation of $\S_2$ is equivalent to $I_1 \oplus \varrho_2(g)$, where $\varrho_2(1)=1$ and $\varrho_2((12))=-1$, where $m$ is odd.
So, in this situation, $Z(H,\phi)$ satisfies that
$Z(H,\phi)(e,v)=\varrho_2(\phi(e,v))=\sgn \phi(e,v)$ if $v \in e$, and $Z(H,\phi)(e,v)=0$ else.
We call $Z(H,\phi)$ the \emph{signed incidence matrix} of $H$ associated with $\phi$.

\begin{coro}\label{sta2}
Let $H$ be a connected $m$-uniform hypergraph on $n$ vertices, let $H_B^\phi$ be a connected $2$-fold covering of $H$, where $\phi$ is a permutation assignment in $\S_2$ on $B_H$, and $m$ is odd.
If the signed incidence matrix $Z(H,\phi)$ has invariant divisors $\hat{d}_1,\ldots,\hat{d}_{\hat{r}}$ over $\Z_m$, then
$$s(H_B^\phi)=s(H) \left(m^{n-\hat{r}} \prod_{i=1}^{\hat{r}}\hat{d}_i\right).$$
\end{coro}

\begin{exm}
Consider the $3$-uniform hypergraph $H$ and permutation assignment $\phi$ in Example \ref{3-uni}.
By Corollary \ref{2-connect}, $H_B^\phi$ is connected as $(B_H, \phi)$ contains a cycle with only one edge assigned by the permutation $(12)$.
Both $Z(H)$ and $Z(H,\phi)$ have invariant divisors $1,1$ over $\Z_3$.
 So $s(H)=3$ by Theorem \ref{stru}, and $s(H_B^\phi)=27$ by Corollary \ref{sta2}.
\end{exm}

We note that Corollary \ref{sta2} does not hold for $m$ being even; see the following example.

\begin{exm}
Consider the $4$-uniform hypergraph $H$ and the permutation assignment $\phi$ in Example \ref{4-uni}(1).
Then $Z(H)$ has invariant divisors $1,1,2$ over $\Z_4$, implying $s(H)=2\cdot 4^2$ by Theorem \ref{stru}.
The matrix $Z(H,\phi)$ has invariant divisors $1,1$ over $\Z_4$, and $Z(H_B^\phi)$ has invariant divisors $1,1,1,1,1$ over $\Z_4$.
We have $s(H_B^\phi)=4^6< (2 \cdot 4^2) \cdot 4^4$.
For the permutation assignment $\psi$ in Example \ref{4-uni}(2),
the matrix $Z(H,\psi)$ has invariant divisors $1,1,2$ over $\Z_4$, and $Z(H_B^\psi)$ has invariant divisors $1,1,1,1,1,1$ over $\Z_4$.
We have $s(H_B^\psi)=4^5<(2 \cdot 4^2)\cdot (2 \cdot 4^3)$.
So, for the two permutation assignments in Example \ref{4-uni}, the equality in Corollary \ref{sta2} does not hold.
\end{exm}

%\section*{Acknowledgement}
%The authors are sincerely thankful to the referees for valuable suggestions to improve the paper.

%\section*{Statements and Declarations}
%
%
%The authors declared that they have no conflicts of interest to this work.

\end{document}